\documentclass[a4paper]{amsart}
\usepackage{amsfonts}
\usepackage{booktabs}
\usepackage{caption}

\usepackage{amsmath}
\usepackage{amssymb}
\usepackage{verbatim}
\usepackage{enumerate}
\usepackage{mathtools}

\newtheorem{MainThm}{Theorem}
\newtheorem{MainCor}[MainThm]{Corollary}
\newtheorem{conj}{Conjecture}

\newtheorem{thm}{Theorem}[section]
\newtheorem{cor}[thm]{Corollary}
\newtheorem{lem}[thm]{Lemma}
\newtheorem{prop}[thm]{Proposition}
\theoremstyle{definition}
\newtheorem{defn}[thm]{Definition}
\theoremstyle{remark}
\newtheorem{rem}[thm]{Remark}
\newtheorem{example}[thm]{Example}

\numberwithin{equation}{section}

\newcommand{\bC}{\mathbb{C}}

\newcommand{\bF}{\mathbb{F}}

\newcommand{\bP}{\mathbb{P}}
\newcommand{\bQ}{\mathbb{Q}}
\newcommand{\bR}{\mathbb{R}}

\newcommand{\bZ}{\mathbb{Z}}

\newcommand\lra{\longrightarrow}

\newcommand\trf{\mathrm{trf}}

\newcommand\colim{\mathrm{colim \,}}
\newcommand\hocolim{\mathrm{hocolim \,}}

\newcommand{\CircNum}[1]{\ooalign{\hfil\raise .00ex\hbox{\scriptsize #1}\hfil\crcr\mathhexbox20D}}

\newcommand{\map}{\mathrm{map}}

\newcommand{\Aut}{\mathrm{Aut}}
\newcommand{\Out}{\mathrm{Out}}

\mathchardef\ordinarycolon\mathcode`\:
\mathcode`\:=\string"8000
\begingroup \catcode`\:=\active
  \gdef:{\mathrel{\mathop\ordinarycolon}}
\endgroup

\usepackage[all]{xy}
\SelectTips{cm}{10}
\CompileMatrices
\title[Cohomology of $\mathrm{Aut}(F_n)$]{Cohomology of automorphism groups of free groups with twisted coefficients}

\author{Oscar Randal-Williams}
\thanks{Supported by ERC Advanced Grant No.\ 228082, the Danish National Research Foundation through the Centre for Symmetry and Deformation, and EPSRC grant EP/M027783/1.}
\email{o.randal-williams@dpmms.cam.ac.uk}
\address{Centre for Mathematical Sciences\\
Wilberforce Road\\
Cambridge CB3 0WB\\
UK}
\date{\today}


\begin{document}

\begin{abstract}
We compute the groups $H^*(\mathrm{Aut}(F_n); M)$ and $H^*(\mathrm{Out}(F_n); M)$ in a stable range, where $M$ is obtained by applying a Schur functor to $H_\bQ$ or $H^*_\bQ$, respectively the first rational homology and cohomology of $F_n$. The answer may be described in terms of stable multiplicities of irreducibles in the plethysm $\mathrm{Sym}^k \circ \mathrm{Sym}^l$ of symmetric powers.

We also compute the stable integral cohomology groups of $\Aut(F_n)$ with coefficients in $H$ or $H^*$.
\end{abstract}
\maketitle

\section{Statement of results}

Galatius \cite{galatius-2006} has proved the remarkable theorem that the natural homomorphisms
$$\Sigma_n \lra \mathrm{Aut}(F_n) \lra \mathrm{Out}(F_n)$$
both induce homology isomorphisms in degrees $2* \leq n - 3$ with integral coefficients. His approach is to model $B\mathrm{Out}(F_n)$ as the space $\mathcal{G}_n$ of graphs of the homotopy type of $\vee^n S^1$, and $B\mathrm{Aut}(F_n)$ as the space $\mathcal{G}^1_n$ of pointed graphs of the same homotopy type. He then produces a natural map from such spaces of graphs to the infinite loop space $Q_0(S^0)$, which he shows has a certain homological connectivity. One consequence of this is that $H^i(\Aut(F_n);\bQ)=0$ for $0 < i \leq \tfrac{n-3}{2}$.

At the same time, Satoh \cite{Satoh, Satoh2} has studied the low dimensional (co)homology of $\mathrm{Aut}(F_n)$ and $\mathrm{Out}(F_n)$ with coefficients in the module $H := H_1(F_n;\bZ)$ given by the abelianisation of $F_n$, and in the dual module $H^* := H^1(F_n;\bZ)$. His methods are those of combinatorial group theory, and proceed by calculation with a presentation of these groups. 

Our goal is to show that the stable cohomology of $\Aut(F_n)$ and $\Out(F_n)$ with twisted coefficients may also be approached with the geometric techniques used by Galatius, along with a little representation theory. The method we shall introduce is quite general and may be applied whenever a suitable Madsen--Weiss-type theorem has been proved; in Appendix \ref{sec:MCG} we will show how to use it to recover a theorem of Looijenga \cite{Looijenga2} on the stable cohomology of mapping class groups with twisted coefficients.

For $\Aut(F_n)$ and $\Out(F_n)$ we will consider cohomology with coefficients in the modules $H^*_\bQ := H^1(F_n; \bQ)$ and $H_\bQ := H_1(F_n;\bQ)$, and more generally with coefficients in $S_\lambda(H^*_\bQ)$ and $S_\lambda(H_\bQ)$, where $S_\lambda(-)$ is the \emph{Schur functor} associated to a partition $\lambda \vdash q$, which we think of as being given by a Young diagram. To define this, recall that to such a partition there is an associated irreducible $\bQ[\Sigma_{q}]$-module $S^\lambda$, the Specht module. For a $\bQ$-vector space $V$ we may consider $V^{\otimes q}$ as a $\bQ[\Sigma_{q}]$-module by permuting the factors, and we may hence form $S_\lambda(V) := \mathrm{Hom}_{\bQ[\Sigma_{q}]}(S^\lambda, V^{\otimes q})$. This construction defines the Schur functor $S_\lambda$. It is a basic result that $S_\lambda(V)$ is an irreducible representation of $GL(V)$.

In fact, our result is best expressed as calculating $H^*(\Aut(F_n);H_\bQ^{\otimes q})$ as a $\bQ[\Sigma_q]$-module. The result for $S_\lambda(H_\bQ)$ may then be extracted as
$$H^*(\Aut(F_n);S_\lambda(H_\bQ)) = \mathrm{Hom}_{\Sigma_q}(S^\lambda , H^*(\Aut(F_n);H_\bQ^{\otimes q})).$$
We write $\bQ^-$ for the sign representation of $\Sigma_q$.

\begin{MainThm}\mbox{}\label{thm:Main1}
\begin{enumerate}[(i)]
\item $H^*(\mathrm{Aut}(F_n); (H^*)^{\otimes q})=0$ for $2* \leq n - q-3$.

\item $H^*(\Aut(F_n); H_\bQ^{\otimes q})=0$ for $2* \leq n-q-3$ if $* \neq q$, and $H^{q}(\Aut(F_n); H_\bQ^{\otimes q}) \otimes \bQ^-$ is the permutation module on the set of partitions of $\{1,2,\ldots,q\}$.
\end{enumerate}
\end{MainThm}

Theorem \ref{thm:Main1} (i) may also be deduced from work of Djament--Vespa \cite{DV}. We believe that Theorem \ref{thm:Main1} (ii) may also be obtained by combining work of Djament \cite{DjamentHodge} and Vespa \cite{VespaExt}. 

However, more important than these particular results is our technique, which is of very general applicability. For example, it can easily be modified to obtain results for $\Out(F_n)$. 

\begin{MainThm}\mbox{}\label{thm:Main2}
\begin{enumerate}[(i)]
\item $H^*(\mathrm{Out}(F_n); (H^*_\bQ)^{\otimes q})=0$ for $2* \leq n - q-3$.

\item $H^*(\Out(F_n); H_\bQ^{\otimes q})=0$ for $2* \leq n-q-3$ if $* \neq q$, and as long as $n \geq 4q+3$ $H^{q}(\Out(F_n) ; H_\bQ^{\otimes q})\otimes \bQ^-$ is the permutation module on the set of partitions of $\{1,2,\ldots,q\}$ having no parts of size 1.
\end{enumerate}
\end{MainThm}

We give tables listing the dimensions of the groups $H^{\vert \lambda \vert}(\Aut(F_n);S_\lambda(H_\bQ))$ and $H^{\vert \lambda \vert}(\Out(F_n);S_\lambda(H_\bQ))$ for $\vert \lambda \vert \leq 6$ in Appendix \ref{sec:tables}.

Each partition of $\{1,2,\ldots,q\}$ may be expressed as a partition of a smaller set with no parts of size 1 along with the set of parts of size 1, which translates to the expression
\begin{equation*}
H^{q}(\Aut(F_n) ; H_\bQ^{\otimes q}) \cong \bigoplus_{\ell \leq q} \mathrm{Ind}_{\Sigma_\ell \times \Sigma_{q-\ell}}^{\Sigma_q} (H^{\ell}(\Out(F_n) ; H_\bQ^{\otimes \ell}) \otimes \bQ^-)
\end{equation*}
as long as $n \geq 4q+3$. Applying $\mathrm{Hom}_{\Sigma_q}(S^\lambda,-)$ and using the Pieri rule gives the pleasant formula
\begin{equation}\label{eq:OutAutTrade}
H^{\vert \lambda \vert}(\Aut(F_n) ; S_\lambda(H_\bQ)) \cong \bigoplus_{\mu \in \rho(\lambda)}  H^{\vert \mu \vert}(\Out(F_n) ; S_\mu(H_\bQ))
\end{equation}
as long as $n \geq 4\vert \lambda \vert+3$, where $\rho(\lambda)$ denotes the set of Young diagrams which may be obtained from $\lambda$ by removing at most one box from each row.

\subsection{Stable plethysm of symmetric powers}

For a Young diagram $\mu$ and integers $k$ and $l$ such that $2\vert \mu \vert \leq kl$, let $(kl-\vert\mu\vert, \mu)$ be the Young diagram obtained by adding a row of length $kl-\vert\mu\vert$ to the top of $\mu$, and let $\nu^{k,l}(\mu)$ denote the multiplicity of the irreducible $GL(V)$-representation $S_{(kl-\vert\mu\vert, \mu)}(V)$ in $\mathrm{Sym}^k(\mathrm{Sym}^l(V))$. Manivel has shown \cite{Manivel} that the numbers $\nu^{k,l}(\mu)$ are increasing and eventually constant functions of both $k$ and $l$, and we write $\nu^\infty(\mu)$ for the stable value. This stable value is attained as soon as $l \geq \mu_1$ and $k \geq \vert \mu \vert$. Using the work of Manivel, we are able to relate these stable multiplicities directly to the cohomology of $\Out(F_n)$, as follows.

\begin{MainThm}\label{thm:Main4}
As long as $n \geq 4\vert \lambda \vert+3$ we have
$$\dim_\bQ H^{\vert \lambda \vert}(\Out(F_n) ; S_{\lambda}(H_\bQ)) = \nu^\infty(\lambda')$$
and as long as $n \geq 2\vert \lambda \vert+3$ we have
$$\dim_\bQ H^{\vert \lambda \vert}(\Aut(F_n) ; S_{\lambda}(H_\bQ)) = \sum_{\mu \in \rho(\lambda)} \nu^\infty(\mu').$$
\end{MainThm}

It follows directly from \cite[Proposition 4.4.1]{Manivel} that $H^{\vert \lambda \vert}(\Out(F_\infty) ; S_{\lambda}(H_\bQ))=0$ if $2 \lambda_1 > \vert \lambda \vert$, or if $2 \lambda_1 = \vert \lambda \vert$ and $\lambda$ does not consist of two rows of equal length.

\subsection{Symmetric and exterior powers}

In order to demonstrate how our theorems may be used, we compute the dimension of the associated cohomology groups for the modules $\wedge^q(H_\bQ)$ and $\mathrm{Sym}^q(H_\bQ)$.

\begin{MainCor}\label{cor:Main}\mbox{}
\begin{enumerate}[(i)]
\item For $n \geq 2q+3$, $H^q(\Aut(F_n) ; \wedge^q(H_\bQ))$ has dimension given by the number of partitions of $q$; for $n \geq 4q+3$, $H^q(\Out(F_n) ; \wedge^q(H_\bQ))$ has dimension given by the number of partitions of $q$ into pieces none of which are 1.

\item Let $q \geq 2$. For $n \geq 2i+q+3$, 
$$H^i(\Aut(F_n) ; \mathrm{Sym}^q(H_\bQ))=H^i(\Out(F_n) ; \mathrm{Sym}^q(H_\bQ))=0.$$
\end{enumerate}
\end{MainCor}
\begin{proof}
The dimension of $H^q(\Aut(F_n) ; \wedge^q(H_\bQ))$ is the multiplicity of the sign representation in $H^q(\Aut(F_n);H_\bQ^{\otimes q})$, which by Theorem \ref{thm:Main1} is the multiplicity of the trivial representation in the permutation module for the set of partitions of $\{1,2,\ldots,q\}$: this is the number of partitions of $q$. 

The dimension of  $H^q(\Aut(F_n) ; \mathrm{Sym}^q(H_\bQ))$ is the multiplicity of the trivial representation in $H^q(\Aut(F_n);H_\bQ^{\otimes q})$, which by Theorem \ref{thm:Main1} is the multiplicity of the sign representation in the permutation module for the set of partitions of $\{1,2,\ldots,q\}$. This may be computed via inner product of characters as $\tfrac{1}{q!}$ times
$$\sum_{\sigma \in \Sigma_q} \#\{\text{partitions $P$ such that $\sigma(P)=P$}\} \cdot \mathrm{sign}(\sigma) = \sum_{\substack{\text{partitions}\\P}} \sum_{\substack{\sigma \in \Sigma_q\\\sigma(P)=P}} \mathrm{sign}(\sigma),$$
but, writing $\Sigma_P \leq \Sigma_q$ for the stabiliser of a partition $P$, we have that $\sum_{\sigma \in \Sigma_P} \mathrm{sign}(\sigma)$ is $\vert \Sigma_P \vert$ times the multiplicity of the sign representation in the trivial representation of $\Sigma_P$, which is zero (as $\Sigma_P$ always contains at least one transposition if $q \geq 2$).

The arguments for $\Out(F_n)$ are identical.
\end{proof}

\subsection{Integral and torsion results}

Finally, our technique can be made to give integral and local information as well. It is not hard to show that $H^{i}(\mathrm{Aut}(F_n);H^*)=0$ for $2i \leq n-4$ (see Proposition \ref{prop:VanishTensorPowers}), but we also have the following.

\begin{MainThm}\label{thm:Main5}\mbox{}
\begin{enumerate}[(i)]
\item $H^{*}(\mathrm{Aut}(F_\infty);H)$ is a free $H^{*}(\mathrm{Aut}(F_\infty);\bZ)$-module (on a single generator in degree 1).

\item For a partition $\lambda \vdash q$ and a prime number $p > q$, $H^{*}(\mathrm{Aut}(F_\infty);S_\lambda(H \otimes \bZ_{(p)}))$ is a free $H^{*}(\mathrm{Aut}(F_\infty);\bZ_{(p)})$-module (on generators in degree $q$, the number of which may be deduced from Theorem \ref{thm:Main1}).
\end{enumerate}
\end{MainThm}
That Theorem \ref{thm:Main5} (ii) might hold was suggested to the author by Aur{\'e}lien Djament upon hearing of Theorem \ref{thm:Main5} (i). We refer to Section \ref{sec:integral} for a description of what we mean by the Schur functor $S_\lambda$ in the $p$-local setting.

One may deduce from Theorem \ref{thm:Main5} (i) that $H^i(\Out(F_n) ; H \otimes \bZ[\tfrac{1}{n-1}])=0$ for $2i \leq n-4$. In Appendix \ref{sec:conj} we describe how the calculation $H^1(\Out(F_n) ; H) \cong \bZ/(n-1)$ for $n \geq 9$ follows from a reasonable-seeming conjecture about spaces of graphs.

\subsection*{Acknowledgements}

This paper is an update to a 2010 preprint \cite{R-WAutOld} in which I proved, inter alia, Corollary \ref{cor:Main} subject to a sequence of conjectures. Since that time, in joint work with Nathalie Wahl \cite{R-WW} we proved Conjecture A (and I explain in this paper how a version of Conjecture B follows from it), and in 2014 S{\o}ren Galatius explained to me a proof of Conjecture C. Thus these conjectural calculations from my 2010 preprint hold.

Recent work of Aur{\'e}lien Djament \cite{DjamentHodge} and Christine Vespa \cite{VespaExt} obtains these calculations by very different means. I was motivated by their results to revisit these techniques and to clarify the status of the conjectures from my 2010 preprint. I would like to thank all of the above named for their interest in, and useful comments on, the content of this note.

\section{An observation regarding homology stability}

The observations of this section are no doubt known to some experts. The groups $\mathrm{Aut}(F_n)$ and $\mathrm{Out}(F_n)$ fit into a more general family of groups denoted $\Gamma_{n,s}$ by Hatcher--Vogtmann \cite{HV}, where $\mathrm{Aut}(F_n) = \Gamma_{n, 1}$ and $\mathrm{Out}(F_n) = \Gamma_{n, 0}$. Classifying spaces for these may be taken to be the spaces $\mathcal{G}_n^s$ of graphs of the homotopy type of $\vee^n S^1$ equipped with $s$ distinct ordered marked points.

Hatcher and Vogtmann \cite{HV} prove that the map $\mathcal{G}_n^s \to \mathcal{G}_{n+1}^{s}$ (defined for $s > 0$) that adds a loop at a particular marked point induces an integral homology isomorphism in degrees $2* \leq n-2$ (and induces a rational homology isomorphism in degrees $5* \leq 4n-10$). Furthermore, the map $\mathcal{G}_n^s \to \mathcal{G}_n^{s-1}$ that forgets a marked point is an integral homology isomorphism in degrees $2* \leq n-3$ (or $2* \leq n-4$ if it is the last marked point).

We can make an immediate observation regarding cohomology with coefficients in $H^*$ from this homology stability result. There is an extension $F_n \to \mathrm{Aut}(F_n) \to \mathrm{Out}(F_n)$, and the corresponding Leray--Hochschild--Serre spectral sequence has two rows. However, as the projection is a homology equivalence in a range of degrees we deduce
\begin{prop}
The groups $H^*(\mathrm{Out}(F_n);H^*)$ are zero for $2* \leq n-6$.
\end{prop}

Similarly, there is a fibration with section $\vee^n S^1 \to \mathcal{G}_n^2 \to \mathcal{G}_n^1$ and the projection map is a homology equivalence in a range of degrees, so
\begin{prop}\label{prop:AutFnHomologyRep}
The groups $H^*(\mathrm{Aut}(F_n);H^*)$ are zero for $2* \leq n-4$.
\end{prop}

More generally, the map $\mathcal{G}_n^{k+1} \to \mathcal{G}_n^1$ has fibre $(\vee^n S^1)^k$ and $\mathrm{Aut}(F_n)$ acts on its homology diagonally. Thus for $k=2$ the Serre spectral sequence has three rows, with
$$E_2^{*, 1} = H^*(\mathrm{Aut}(F_n); H^* \oplus H^*) = 0 \quad \text{for $2* \leq n-4$ by Proposition \ref{prop:AutFnHomologyRep}}$$
and
$$E_2^{*, 2} = H^*(\mathrm{Aut}(F_n); H^* \otimes H^*).$$
Using the fact that the projection map is a homology equivalence in degrees $2* \leq n-2$, we deduce that $H^*(\mathrm{Aut}(F_n); H^* \otimes H^*)=0$ for $2* \leq n - 8$. Continuing in this way for higher $k$, we establish the following proposition.
\begin{prop}\label{prop:VanishTensorPowers}
For all $q \geq 1$, $H^*(\mathrm{Aut}(F_n); (H^*)^{\otimes q})$ is zero for $2* \leq n - 4q$.
\end{prop}

\section{Homology stability with coefficient systems}

By Galatius' theorem, the groups $\mathrm{Aut}(F_n)$ are closely related to the symmetric groups, but also share many properties with mapping class groups of surfaces. These three families of groups are known to exhibit homological stability for integral homology, but symmetric groups and mapping class groups also exhibit homological stability for certain systems of coefficients, those of ``finite degree", a notion that is originally due to Dwyer \cite{DwyerHomStab} in his study of homological stability for general linear groups with coefficient systems.

This notion of degree may be formalised, in the context of free groups, as follows. 
\begin{defn}
Let $\mathfrak{Gr}$ denote the category whose objects are the finitely generated free groups, and where a morphism from $G$ to $H$ is given by a pair
$$(f : G \to H, X \leq H)$$
consisting of an injective group homomorphism $f$ and a finitely-generated (free) subgroup $X \leq H$ such that $H = f(G) * X$.

A \emph{coefficient system} is a covariant functor $V : \mathfrak{Gr} \to \mathbf{Ab}$ to the category of abelian groups. We declare the constant functors to be \emph{polynomial of degree 0}, and more generally we declare a functor $V$ to be polynomial of degree $\leq k$ if
\begin{enumerate}[(i)]
\item $V$ sends the the canonical morphism $s_G := (\mathrm{inc} : G \to G*\bZ, \bZ)$ to an injection, and

\item the new coefficient system $G \mapsto \mathrm{Coker}(s_G : V(G) \to V(G * \bZ))$ is polynomial of degree $\leq k-1$.
\end{enumerate}
\end{defn}

In a 2010 preprint \cite{R-WAutOld}, we conjectured (based on the analogy with general linear groups \cite{DwyerHomStab} and mapping class groups \cite{Ivanov}) that the groups $H_*(\mathrm{Aut}(F_n) ; V(F_n))$ should exhibit homological stability in degrees $2* \leq n-k-2$ when $V$ is a polynomial coefficient system of degree $\leq k$. Since then, in joint work with Wahl \cite{R-WW} we have established a quite general homological stability theorem with polynomial coefficients, and using the highly-connected simplicial complexes of \cite{HV2} it applies in this case. The result obtained is as follows.

\begin{thm}[Randal-Williams--Wahl \cite{R-WW}]\label{thm:RWW}
If $V$ is a polynomial coefficient system of degree $\leq k$, then the natural map
$$H_*(\mathrm{Aut}(F_n);V(F_n)) \lra H_*(\mathrm{Aut}(F_{n+1});V(F_{n+1}))$$
induces an epimorphism for $2* \leq n-k-1$ and an isomorphism for $2* \leq n-k-3$.
\end{thm}

One advantage of the category $\mathfrak{Gr}$ over the more na{\"i}ve category $\mathfrak{gr}$ of finitely generated free groups and injective homomorphisms is that there are more functors out of it.

\begin{defn}\mbox{}
\begin{enumerate}[(i)]
\item Let $H : \mathfrak{Gr} \to \mathbf{Ab}$ be the coefficient system sending $G$ to $H_1(G;\bZ) = G^{ab}$, and sending a morphism $(f : G \to H, X)$ to $f_*$. There is an exact sequence
$$0 \lra H(G) \overset{H(s_G)}\lra H(G * \bZ) \lra \bZ \lra 0$$
so this coefficient system is polynomial of degree 1.
\item Let $H^* : \mathfrak{Gr} \to \mathbf{Ab}$ be the coefficient system sending $G$ to $H^1(G;\bZ) = \mathrm{Hom}(G^{ab}, \bZ)$ , and sending a morphism $(f : G \to H, X)$ to the linear dual of
$$H_1(H;\bZ) = H_1(f(G) * X ; \bZ) \lra H_1(f(G);\bZ) \cong H_1(G;\bZ).$$
(Note that this does \emph{not} define a functor on $\mathfrak{gr}$!) There is an exact sequence
$$0 \lra H^*(G) \overset{H^*(s_G)}\lra H^*(G * \bZ) \lra \bZ \lra 0$$
so this coefficient system is polynomial of degree 1.
\end{enumerate}
\end{defn}

\begin{rem}
Not only can more coefficient systems be defined on $\mathfrak{Gr}$ than on $\mathfrak{gr}$, but it follows from recent work of Djament--Vespa \cite{DV} that the most homologically interesting functors \emph{must} be defined here: they show that if $V$ is a polynomial coefficient system which is {reduced} (i.e. $V(\{e\})=0$) and {factors through $\mathfrak{gr}$}, then
$$\underset{n \to \infty}\colim H_*(\Aut(F_n) ; V(F_n))=0.$$
The coefficient system $H^*$ is reduced, but does not factor through $\mathfrak{gr}$: indeed, Satoh \cite{Satoh} has shown that $\underset{n \to \infty}\colim H_1(\Aut(F_n) ; H^*)=\bZ$, and we will show how to recover (the dual version of) this in Section \ref{sec:integral}.
\end{rem}

There are some easy consequences of the definition of polynomiality which allow us to compute or estimate degrees of coefficient systems. Firstly, a summand  of a polynomial functor of degree $\leq k$ is again polynomial of degree $\leq k$, and an extension of two polynomial functors of degree $\leq k$ is again polynomial of degree $\leq k$. Secondly, if $V$ and $W$ are coefficient systems which are polynomial of degree $k$ and $\ell$ respectively, and if they both take values in flat $\bZ$-modules, then $V \otimes_\bZ W$ is polynomial of degree $\leq (k+\ell)$. More generally, if $\bF$ is a field and $V,  W : \mathfrak{Gr} \to \bF\text{-mod} \to \mathbf{Ab}$ are coefficient systems which are polynomial of degree $k$ and $\ell$ respectively and factor through the category of $\bF$-modules, then $V \otimes_\bF W$ is polynomial of degree $\leq (k+\ell)$.

We will often be interested in cohomology rather than homology, for which we will use the following result of Universal Coefficient-type, which is surely standard but for which we could not find a reference.

\begin{lem}\label{lem:UCT}
Let $G$ be a group, $R$ be a PID, and $M$ a left $R[G]$-module, and write $M^* = \mathrm{Hom}_R(M, R)$, a right $R[G]$-module. There is a natural short exact sequence
$$0 \lra \mathrm{Ext}_R^1(H_{i-1}(G;M), R) \lra H^i(G; M^*) \lra \mathrm{Hom}_R(H_i(G;M), R) \lra 0.$$
\end{lem}
\begin{proof}
Analogous to \cite[Proposition 7.1]{Brown}. Let $P_\bullet \to R$ be a projective right $R[G]$-module resolution, so $C^\bullet := \mathrm{Hom}_{R[G]}(P_\bullet, M^*)$ is a cochain complex of $R$-modules which computes $H^*(G; M^*)$. Writing
$$C^\bullet = \mathrm{Hom}_{R}(P_\bullet, M^*)^G \cong \mathrm{Hom}_{R}(P_\bullet \otimes_R M, R)^G \cong \mathrm{Hom}_R(P_\bullet \otimes_{R[G]} M, R)$$
and applying the Universal Coefficient Theorem (e.g.\ \cite[p. 114]{CE}), using the fact that a submodule of a free module over a PID is free, gives the desired sequence.
\end{proof}

Using the stability theorem and Lemma \ref{lem:UCT}, we may improve Proposition \ref{prop:VanishTensorPowers} to the following.

\begin{cor}\label{cor:ThmAi}
For all $q \geq 1$, $H^*(\mathrm{Aut}(F_n); (H^*)^{\otimes q})=0$ for $2* \leq n - q-3$.
\end{cor}

This proves Theorem \ref{thm:Main1} (i). We remark that Theorem \ref{thm:Main1} (i) also follows from \cite[Th{\'e}or{\`e}me 1]{DV} and Lemma \ref{lem:UCT}, because the functor $H^{\otimes q} : \mathfrak{Gr} \to \mathfrak{gr} \to \mathbf{Ab}$ is polynomial and vanishes on the trivial group (for $q \geq 1$), so by that theorem $H_*(\Aut(F_n);H^{\otimes q})$ vanishes in the stable range.

Finally, we record the stability range for cohomology with coefficients in $S_\lambda(H_\bQ)$ with $\lambda \vdash q$. This follows from $S_\lambda(H_\bQ) \cong S_\lambda(H_\bQ^*)^*$, that $S_\lambda(H_\bQ^*)$ is polynomial of degree $\leq q$ as it is a summand of $(H_\bQ^{*})^{\otimes q}$, the stability theorem, and Lemma \ref{lem:UCT}.

\begin{cor}
The groups $H^i(\Aut(F_n);S_\lambda(H_\bQ))$ are independent of $n$ for $2i \leq n-q-3$.
\end{cor}

\section{Graphs labeled by a space $X$}

Let $(X,x_0)$ be a based space. Let us write $\mathcal{G}^X$ for the topological category whose objects are finite sets inside $\bR^\infty$, and whose morphisms from $S \subset \bR^\infty$ to $T \subset \bR^\infty$ consist of a real number $t > 0$ and a graph $\Gamma \subset [0,t] \times \bR^\infty$ with leaves $\{0\} \times S$ and $\{t\} \times T$, equipped with a continuous map $f : (\Gamma, S \cup T) \to (X, x_0)$. This may be given a topology following \cite{galatius-2006}. Choosing once and for all an embedding $\{1,2,\ldots,s\} \subset \bR^\infty$ the space $\mathcal{G}_n^s(X)$ may be defined as the subspace of $\mathcal{G}^X(\emptyset, \{1,2,\ldots,s\})$ consisting of those graphs which are connected and homotopy equivalent to $\vee^n S^1$.

Let $\mathcal{G}_n^s(X)_* \subset \mathcal{G}_n^s(X) \times \bR^{\infty+1}$ be the subspace given by tuples $(t, \Gamma, f; p)$ where $p \in \Gamma$. Forgetting the point $p$ gives a map
$$\pi : \mathcal{G}_n^s(X)_* \lra \mathcal{G}_n^s(X)$$
which is a fibration with fibre over the point $(t, \Gamma, f) \in \mathcal{G}_n^s(X)$ given by the graph $\Gamma$. Furthermore, sending $(t, \Gamma, f, p)$ to $f(p) \in X$ defines a map
$$e : \mathcal{G}_n^s(X)_* \lra X.$$

As the map $\pi$ is a fibration with homotopy-finite fibres, it admits a Becker--Gottlieb transfer map $\trf_\pi : \Sigma^\infty_+ \mathcal{G}_n^s(X) \to \Sigma^\infty_+\mathcal{G}_n^s(X)_*$. Composing this with $\Sigma^\infty_+ e$ and taking the adjoint gives a map
$$\tau_n^s : \mathcal{G}_n^s(X) \lra Q_{1-n}(X_+)$$
to the free infinite loop space on $X$, landing in the path component indexed by $1-n$, the Euler characteristic of $\vee^n S^1$. This map is natural in $X$, and factors through the map $\tau^0_n : \mathcal{G}_n(X) \to Q_{1-n}(X_+)$ up to homotopy.

\begin{thm}\label{thm:Xlim}
The induced map
$$\tau_\infty^1 : \underset{n \to \infty} \hocolim \mathcal{G}^1_n(X) \lra Q_0(X_+)$$
is an integral homology equivalence as long as $X$ is path-connected.
\end{thm}

The point of this theorem is that the right-hand side may be computed: the rational cohomology of $Q_{1-n}(X_+)$ may be described compactly as $S^*(\widetilde{H}^*(X;\bQ))$, the free graded-commutative algebra on the reduced cohomology of $X$.

\begin{proof}[Proof sketch]
Following \cite{galatius-2006}, and keeping track of the role that the maps to $X$ play, one may show that $B\mathcal{G}^X \simeq Q(S^1 \wedge X_+)$. Furthermore, under this equivalence the composition
$$\mathcal{G}_n(X) \subset \mathcal{G}^X(\emptyset, \emptyset) \lra \Omega B\mathcal{G}^X \simeq Q(X_+)$$
is weakly homotopic to $\tau^0_n$ by the analogue (with maps to $X$) of the discussion in Section 5.3 of \cite{galatius-2006}.

Now we claim that the natural map
$$\underset{n \to \infty} \hocolim \mathcal{G}^1_n(X) \lra \Omega_{[\emptyset, \{0\}]} B\mathcal{G}^X$$
is an integral homology equivalence as long as $X$ is path-connected. This may be proved following \cite{galatius-2006} and also using ideas from \cite{GR-W}. The crucial point is to consider the category $\mathcal{G}^X_\bullet$ where objects $S$ contain the origin $0 \in \bR^\infty$ and morphisms $\Gamma$ contain the interval $[0,t] \times \{0\}$ on which the function $f$ is constant. The inclusion $\mathcal{G}^X_\bullet \to \mathcal{G}^X$ may be seen to induce an equivalence on classifying spaces as in \cite[Lemma 4.6]{GR-W}. One then considers the subcategory $\mathcal{G}^X_{\bullet,c} \subset \mathcal{G}^X_\bullet$ having only object $\{0\}$ and in which the morphisms are required to be connected. We claim that this inclusion induces an equivalence on classifying spaces. To make morphisms connected use a move similar to that of \cite[Lemma 4.24]{galatius-2006} which connects an arbitrary path component to the standard stick $\bR \times \{0\} \subset \bR \times \bR^{\infty}$ (this requires $X$ to be path-connected). To reduce to a single object use a move similar to that of \cite[Section 4]{GR-W}  to make objects consist of a single point, then isotope them into standard position as in \cite[Proposition 4.26]{GR-W}.

The category $\mathcal{G}^X_{\bullet,c}$ is therefore a monoid, and, as we can slide edges along the standard interval $[0,t] \times \{0\}$, it is a homotopy commutative monoid. We may thus apply the group-completion theorem to it, showing that
$$\hocolim \mathcal{G}^X_{\bullet,c}(\{0\}, \{0\}) \lra \Omega B\mathcal{G}^X_{\bullet,c},$$
is a homology isomorphism, where the homotopy colimit is formed by left multiplication with a connected graph of genus 1. By precomposing with a morphism $\emptyset \leadsto \{0\} \in \mathcal{G}^X$ given by an interval, this map is easily compared with that in the statement of the theorem.
\end{proof}

The main technical result we will require is the following homological stability theorem for the spaces $\mathcal{G}_{n}^s(X)$. We will deduce it from (two) arguments of Cohen--Madsen in the analogous situation of surfaces with maps to a background space.

\begin{thm}\label{thm:Xstab}
Suppose that $X$ is simply-connected.
\begin{enumerate}[(i)]
\item The map $\mathcal{G}^1_n(X) \to \mathcal{G}^1_{n+1}(X)$ induces a homology isomorphism in degrees $2* \leq n - 3$.

\item The map $\mathcal{G}^1_n(X) \to \mathcal{G}_{n}(X)$ induces an isomorphism on homology with $\bZ[\frac{1}{n-1}]$-module coefficients in degrees $2* \leq n - 3$.
\end{enumerate}
\end{thm}

\begin{proof}
For part (i) we follow the argument of Cohen and Madsen \cite{CM}. There is a map of homotopy fibre sequences
\begin{equation*}
\xymatrix{
\mathrm{map}_*(\vee^n S^1, X) \ar[r] \ar[d] & \mathcal{G}^1_n(X) \ar[d] \ar[r] & \mathcal{G}^1_n(*) \ar[d]\\
\mathrm{map}_*(\vee^{n+1} S^1, X) \ar[r] & \mathcal{G}^1_{n+1}(X) \ar[r] & \mathcal{G}^1_{n+1}(*)
}
\end{equation*}
so it is enough to show that $\Aut(F_n) \circlearrowleft H_i(\mathrm{map}_*(\vee^n S^1, X) ; \bZ)$ is part of a polynomial coefficient system of degree $\leq i$. In this case the map of Serre spectral sequences will induce an isomorphism on $E^2_{s,t}$ for $2s \leq n-t-3$ and an epimorphism for $2s \leq n-t-1$; in particular it induces an isomorphism for $2(s+t) \leq n-3$ and an epimorphism for $2(s+t) \leq n-1$. It follows from the spectral sequence comparison theorem that the map $H_*(\mathcal{G}^1_n(X)) \to H_*(\mathcal{G}^1_{n+1}(X))$ in an isomorphism in degrees $2* \leq n-3$ and an epimorphism in degrees $2* \leq n-1$.

We define a coefficient system $V_i^X : \mathfrak{Gr} \to \mathbf{Ab}$ on objects by
$$V_i^X(G) := H_i(\map_*(BG, X);\bZ)$$
and on a morphism $(f: G \to H, X)$ we use precomposition by
$$BH = B(f(G) * X) \lra B(f(G)) \cong BG.$$
This defines a coefficient system, and since $B(G * \bZ) \to BG$ is split surjective the stabilisation maps $V_i^X(s_G)$ are all split injective. The coefficient system $V_0^X$ agrees with the constant coefficient system $\bZ$, as $X$ has been assumed to be simply-connected: in particular it is polynomial of degree 0.

Let us suppose for an induction that $V_j^X$ has degree $\leq j$ for all $j < i$. Consider the homotopy fibre sequence
$$\mathrm{map}_*(BG, X) \overset{\iota}\lra \mathrm{map}_*(B(G * \bZ), X) \lra \map_*(B\bZ, X),$$
where the map $\iota$ induces $V_i^X(s_G)$ on $i$th homology. As $\iota$ is split injective, the Serre spectral sequence for this fibration (which is over a path-connected base) collapses, and we find that $\mathrm{Coker}(V_i^X(G) \to V_i^X(G * \bZ))$ has a filtration with associated graded
$$\{H_{i-j}(\Omega X ; V_{j}^X(G))\}_{j=0}^{i-1}.$$
Each $V_{j}^X(-)$ is polynomial of degree $\leq i-1$ so $H_{i-j}(\Omega X ; V_{j}^X(-))$ is too; as degree is preserved under extensions it follows that $\mathrm{Coker}(V_i^X(G) \to V_i^X(G * \bZ))$ has degree $\leq i-1$, hence $V_i^X$ has degree $\leq i$.

For part (ii), we follow a different argument of Cohen and Madsen \cite{CM2}. We have a diagram
\begin{equation*}
\xymatrix{
 & \vee^n S^1 \ar[d] \\
\mathcal{G}_n^1(X) \ar[r] \ar[rd]^-f& \mathcal{G}_n(X)_* \ar[d]^-\pi \ar[r]^-e & X\\
 & \mathcal{G}_n(X) \ar[r]^-{\tau_n^0}& Q_{1-n}(X_+)
}
\end{equation*}
in which the row and column are fibrations, and the map $\tau_n^0$ is the adjoint to the map of spectra $\Sigma^\infty_+ \mathcal{G}_n(X) \overset{\trf_\pi}\to \Sigma^\infty_+ \mathcal{G}_n(X)_* \overset{e}\to \Sigma^\infty_+ X$ given by the composition of the Becker--Gottlieb transfer for the map $\pi$ followed by the map $e$ given by evaluating the map to $X$ at the marked point.

By part (i), Theorem \ref{thm:Xlim}, and the analogue with maps to $X$ of the discussion in Section 5.3 of \cite{galatius-2006}, the composition $\tau_n^0 \circ f \simeq \tau^1_n : \mathcal{G}_n^1(X) \to Q_{1-n}(X_+)$ is an isomorphism in degrees $2* \leq n-3$, and in particular $\tau^0_n$ is surjective on homology in this range. It follows from Leray--Hirsch that the map
$$e \times (\tau_n^0 \circ \pi): \mathcal{G}_n(X)_* \lra X \times Q_{1-n}(X_+)$$
is also an isomorphism in this range.

Now we have the commutative diagram
\begin{equation*}
\xymatrix{
H_*(\mathcal{G}_n(X)) \ar@{=}[d] \ar[r]^-{\trf_\pi}& H_*(\mathcal{G}_n(X)_*) \ar[r]_-\sim^-{e \times \tau^0_n \circ \pi}& H_*(X \times Q_{1-n}(X_+))\\
H_*(\mathcal{G}_n(X)) \ar[r]^-{\tau^0_n}& H_*(Q_{1-n}(X_+)) \ar[r]^-{\Delta}& H_*(Q_{1-n}(X_+) \times Q_{1-n}(X_+)) \ar[u]_-{\sigma \otimes \mathrm{Id}},
}
\end{equation*}
where $\sigma$ denotes the homology suspension. The transfer map $\trf_\pi$ is (split) injective with $\bZ[\frac{1}{n-1}]$-module coefficients, so in degrees $2* \leq n-3$ the map $\tau_n^0$ is injective too. Hence $\tau^0_n$ is an isomorphism in this range, so $f$ is too.
\end{proof}

\section{Proof of the main theorems}

\subsection{Recollections on Schur--Weyl duality}

We shall need a small amount of representation theory, but nothing beyond e.g. the first few parts of Chapter 9 of \cite{Procesi}. Recall that for a vector space $V$ (over $\bQ$) and a partition $\lambda \vdash q$ we have defined $S_\lambda(V)$ as $\mathrm{Hom}_{\Sigma_q}(S^\lambda, V^{\otimes q})$. If $V$ is finite-dimensional then as the Specht modules $S^\lambda$ give a complete set of irreducible representations of $\Sigma_q$ we may write
$$V^{\otimes q} \cong \bigoplus_{\lambda \vdash q} S_\lambda(V) \otimes S^\lambda$$
as $GL(V) \times \Sigma_q$-modules. The $S_\lambda(V)$ are non-zero as long as $\dim_\bQ V \geq q$, in which case they are irreducible $GL(V)$-modules.

\subsection{Labeled partitions}

Let us fix a sequence of $\bQ$-vector spaces $W_1, W_2, \ldots$, and consider 
$$F(V):=\mathrm{Sym}^*(W_1 \otimes \mathrm{Sym}^1(V) \oplus W_2 \otimes \mathrm{Sym}^2(V) \oplus W_3 \otimes \mathrm{Sym}^3(V) \oplus \cdots).$$
If we consider this as a functor landing in graded vector spaces, by letting $V$ have degree 1, then the homogenous pieces of $F$ are each polynomial functors of $V$ (see \cite[\S 9.7]{Procesi} for this notion). Thus the $\Sigma_q$-module $\mathrm{Hom}_{GL(V)}(V^{\otimes q}, F(V))$ is independent of $V$ as long as $\dim_\bQ V \geq q$. We wish to identify this $\Sigma_q$-module.

Choose bases $\Omega_i$ for the vector spaces $W_i$, and let $\mathcal{P}_q(\Omega_1, \Omega_2, \ldots)$ be the set of partitions $P$ of $\{1,2,\ldots,q\}$ equipped with a labeling of each part $X$ of size $i$ with an element $\ell_X \in \Omega_i$. For each such datum $(P, \ell)$ there is a map of $GL(V)$-modules
\begin{align*}
\phi_{P, \ell} : V^{\otimes q} &\lra \mathrm{Sym}^*(W_1 \otimes \mathrm{Sym}^1(V) \oplus W_2 \otimes \mathrm{Sym}^2(V)  \oplus \cdots)\\
v_1 \otimes \cdots \otimes v_q &\longmapsto \prod_{X \subset P} \left[\ell_X \cdot \prod_{i \in X} v_i \right].
\end{align*}
If $\sigma \in \Sigma_q$ then $\phi_{P, \ell} \circ \sigma = \phi_{\sigma(P, \ell)}$, giving a $\Sigma_q$-equivariant map
$$\phi : \bQ\{\mathcal{P}_q(\Omega_1, \Omega_2, \ldots)\} \lra \mathrm{Hom}_{GL(V)}(V^{\otimes q}, F(V)).$$

\begin{prop}\label{prop:IdentifySymSym}
The map $\phi$ is an isomorphism as long as $\dim_\bQ V \geq q$.
\end{prop}

\begin{proof}
Let us call an orbit of $\Sigma_q$ acting on $\mathcal{P}_q(\Omega_1, \Omega_2, \ldots)$ a \emph{type of partition}. It consists of a partition $\lambda \vdash q$ along with an unordered list of labels $L_i$ in $\Omega_i$ for the parts of size $i$, and we write $\mathcal{P}(\lambda; L_1, L_2, \ldots)$ for this orbit.

The source of $\phi$ splits as a direct sum of (cyclic) modules $\bQ\{\mathcal{P}(\lambda; L_1, L_2, \ldots)\}$ one for each type of partition. For a type of partition $\mathcal{P}(\lambda; L_1, L_2, \ldots)$ and a $\omega \in \Omega_i$, let us write $\vert \omega \vert=i$ and $a_\omega$ for the number of parts labeled by $\omega$. The target of $\phi$ splits as a direct sum of terms
$$M(\lambda; L_1, L_2, \ldots) := \mathrm{Hom}_{GL(V)}\left(V^{\otimes q}, \bigotimes_{\omega \in \cup_{i} \Omega_i}\mathrm{Sym}^{a_{\omega}}(\{\omega\} \otimes \mathrm{Sym}^{\vert \omega \vert}(V))\right)$$
one for each type of partition. The map $\phi$ restricts to a $\bQ[\Sigma_q]$-module map
$$\phi_{\mathcal{P}(\lambda; L_1, L_2, \ldots)}: \bQ\{\mathcal{P}(\lambda; L_1, L_2, \ldots)\} \lra M(\lambda; L_1, L_2, \ldots),$$
and it is enough to show that each of these is an isomorphism.

As $GL(V)$-representations we have a surjection
$$\bigotimes_{\omega \in \cup_{i} \Omega_i}(\{\omega\} \otimes V^{\otimes \vert \omega \vert})^{\otimes a_{\omega}} \lra  \bigotimes_{\omega \in \cup_{i} \Omega_i}\mathrm{Sym}^{a_{\omega}}(\{\omega\} \otimes \mathrm{Sym}^{\vert \omega \vert}(V))$$
which is split by the standard symmetrisers. Furthermore if we choose a $(P,\ell) \in \mathcal{P}(\lambda; L_1, L_2, \ldots)$ then we obtain an isomorphism
\begin{align*}
V^{\otimes q} &\overset{\sim}\lra \bigotimes_{\omega \in \cup_{i} \Omega_i}(\{\omega\} \otimes V^{\otimes \vert \omega \vert})^{\otimes a_{\omega}}\\
v_1 \otimes \cdots \otimes v_q &\longmapsto \bigotimes_{X \subset P} \left(\ell_X \otimes \bigotimes_{i \in X} v_i \right)
\end{align*}
(where these terms must be suitably permuted to be put in the right form).

In total this identifies $M(\lambda; L_1, L_2, \ldots)$ with a summand of $\mathrm{Hom}_{GL(V)}(V^{\otimes q}, V^{\otimes q})$, which by the first fundamental theorem of invariant theory for $GL(V)$ is a free left $\bQ[\Sigma_q]$-module generated by the identity map of $V^{\otimes q}$ (as long as $\dim_\bQ V \geq q$). Thus the $\bQ[\Sigma_q]$-module $M(\lambda; L_1, L_2, \ldots)$ may be identified with the cyclic left submodule of $\bQ[\Sigma_q]$ generated by the idempotent $\gamma:=\tfrac{1}{\vert S \vert}\sum_{\sigma \in S} \sigma$, where $S \leq \Sigma_q$ is the stabiliser of $(P,\ell) \in \mathcal{P}_q(\Omega_1, \Omega_2, \ldots)$.

Under this identification we have $\phi_{\mathcal{P}(\lambda; L_1, L_2, \ldots)}(P, \ell)=\gamma$, so $\phi_{\mathcal{P}(\lambda; L_1, L_2, \ldots)}$ is surjective. On the other hand, we claim that the submodule $\bQ[\Sigma_q]\cdot\gamma \leq \bQ[\Sigma_q]$ is isomorphic to $\bQ[\Sigma_q/S]$, and hence has dimension the size of the orbit $\mathcal{P}(\lambda; L_1, L_2, \ldots)$. This implies that $\phi_{\mathcal{P}(\lambda; L_1, L_2, \ldots)}$ is an isomorphism. To prove the claim, note that the surjective module map $\bQ[\Sigma_q]\cdot\gamma \oplus \bQ[\Sigma_q]\cdot(1-\gamma)=\bQ[\Sigma_q] \to \bQ[\Sigma_q/S]$ sends $1-\gamma$ to 0, so gives a surjection $\bQ[\Sigma_q]\cdot\gamma \to \bQ[\Sigma_q/S]$. On the other hand the module map
$$\sum a_g gS \mapsto \sum a_g g \cdot \gamma : \bQ[\Sigma_q/S] \lra \bQ[\Sigma_q]\cdot\gamma$$
is well-defined and surjective.
\end{proof}

\subsection{Proof of Theorem \ref{thm:Main1}}

We have already proved Theorem \ref{thm:Main1} (i) in Corollary \ref{cor:ThmAi}. For  Theorem \ref{thm:Main1} (ii), first choose a functorial model for Eilenberg--MacLane spaces $K(-, n)$, then fix a finite-dimensional $\bQ$-vector space $V$ and consider the fibration
$$K(H^* \otimes_\bZ V^*, 1) \simeq \mathrm{map}_*(\vee^n S^1, K(V^*, 2)) \lra \mathcal{G}_n^1(K(V^*, 2)) \lra  \mathcal{G}_n^1(*) \simeq B\mathrm{Aut}(F_n)$$
%with section (given by taking the constant map to the basepoint) 
and its associated Serre spectral sequence with $\bQ$-coefficients
\begin{equation}\label{eq:sseqV}
E_2^{p,q} := H^p(\mathrm{Aut}(F_n); \wedge^q(H_\bQ \otimes V)) \Longrightarrow H^{p+q}(\mathcal{G}_n^1(K(V^*, 2));\bQ).
\end{equation}
The action of $GL(V)$ on $K(V^*, 2)$, and hence on the fibration above, make this into a spectral sequence of $\bQ[GL(V)]$-modules.

The proof of the following key lemma is close to an argument communicated to the author by S{\o}ren Galatius to prove \cite[Conjecture C]{R-WAutOld}. It was his argument that led us to think along the lines necessary to prove Theorem \ref{thm:Main1}.

\begin{lem}\label{lem:collapse}
The spectral sequence \eqref{eq:sseqV} collapses.
\end{lem}
\begin{proof}
The action of the scalars $\bQ^\times \leq GL(V)$ on $V$ makes it a spectral sequence of $\bQ[\bQ^\times]$-modules, and the action of $\bQ^\times$ on the $\bQ$-vector space $\wedge^q(H_\bQ \otimes V)$ is with \emph{weight $q$}, i.e.\ $u \in \bQ^\times$ acts by scalar multiplication by $u^q$. As $\bQ$ has characteristic zero, distinct weights make $\bQ$ into distinct irreducible $\bQ[\bQ^\times]$-modules. Thus there can be no $\bQ[\bQ^\times]$-module maps between different rows of this spectral sequence, so it collapses.
\end{proof}

Furthermore, this argument identifies $H^p(\mathrm{Aut}(F_n); \wedge^q_\bQ(H_\bQ \otimes V))$ with the subspace $H^{p+q}(\mathcal{G}_n^1(K(V^*, 2));\bQ)^{(q)}$ of $H^{p+q}(\mathcal{G}_n^1(K(V^*, 2));\bQ)$ on which $\bQ^\times$ acts with weight $q$.

\begin{lem}\label{lem:Abut}
If $2(p+q) \leq n-3$ then
$H^{p+q}(\mathcal{G}_n^1(K(V^*, 2));\bQ)^{(q)}$ is zero unless $p=q$, in which case it is isomorphic to the degree $2q$ part of $\mathrm{Sym}^*(\mathrm{Sym}^{* > 0}(V[2]))$.
\end{lem}
\begin{proof}
By Theorem \ref{thm:Xlim} and Theorem \ref{thm:Xstab}, the map 
$$\tau^1_n: \mathcal{G}_n^1(K(V^*, 2)) \lra Q_{1-n}(K(V^*, 2)_+)$$
is an isomorphism on cohomology in degrees $2* \leq n-3$. With coefficients of characteristic zero, $\widetilde{H}^*(K(V^*, 2);\bQ) \cong \mathrm{Sym}^{*>0}(V[2])$, and taking the free infinite loop space has the effect of forming the free graded-commutative algebra, so in this case the symmetric algebra. The action of $\bQ^\times$ on $\mathrm{Sym}^*(\mathrm{Sym}^{* > 0}(V[2]))$ is with weight $q$ precisely in degree $2q$. 
\end{proof}

\begin{cor}
If $2(p+q) \leq n-3$ then 
$H^p(\mathrm{Aut}(F_n); \wedge^q_\bQ(H_\bQ \otimes_\bQ V))$ is zero unless $p=q$, in which case it is isomorphic to the degree $2q$ part of $\mathrm{Sym}^*(\mathrm{Sym}^{* > 0}(V[2]))$.
\end{cor}

This proves the vanishing part of Theorem \ref{thm:Main1} (ii). We shall now use the fact that the identifications made so far are functorial in $V$, and so are in particular $GL(V)$-equivariant and can be decomposed into irreducible $GL(V)$-modules. The following is a standard consequence of Schur--Weyl duality, but we explain its proof anyway.

\begin{lem}\label{lem:SymSplit}
As a $GL(H_\bQ) \times GL(V)$-representation,
$$\wedge^q_\bQ(H_\bQ \otimes V) \cong \bigoplus_{\vert \lambda \vert = q} S_\lambda(H_\bQ) \otimes S_{\lambda'}(V),$$
where $\lambda'$ denotes the conjugate (i.e.\ transpose) Young diagram to $\lambda$.
\end{lem}
\begin{proof}
The left-hand side is the $\Sigma_q$-invariants in the $GL(H_\bQ) \times GL(V) \times \Sigma_q$-module $(H_\bQ \otimes_\bQ V)^{\otimes q} \otimes \bQ^- \cong H_\bQ^{\otimes q} \otimes V^{\otimes q}\otimes \bQ^-$. By our definition of Schur functors we have an isomorphism of $GL(V) \times \Sigma_q$-modules
$$V^{\otimes q} \cong \bigoplus_{\lambda \vdash q} S_\lambda(V) \otimes S^\lambda$$
where $S^\lambda$ is the Specht module. Similarly $H_\bQ^{\otimes q} \cong \bigoplus_{\mu \vdash q} S_\mu(H_\bQ) \otimes S^\mu$, so tensoring them together and taking $\Sigma_q$-invariants, using $(S^\mu \otimes S^\lambda \otimes \bQ^-)^{\Sigma_q} = \bQ^{\delta_{\mu \lambda'}}$, the result follows.
\end{proof}

Putting the above together, we have an isomorphism of $GL(V)$-modules
$$\bigoplus_{\vert \lambda \vert = q} H^q(\Aut(F_n) ; S_\lambda(H_\bQ)) \otimes S_{\lambda'}(V) \cong [\mathrm{Sym}^*(\mathrm{Sym}^{* > 0}(V[2]))]_{2q}.$$
Choosing $V$ to be at least $q$-dimensional, the $S_\mu(V)$ are then distinct non-zero irreducible $GL(V)$-modules, so by Schur's lemma applying $\mathrm{Hom}_{GL(V)}(S_{\lambda'}(V),-)$ gives
$$H^q(\Aut(F_n) ; S_\lambda(H_\bQ)) \cong \mathrm{Hom}_{GL(V)}(S_{\lambda'}(V),\mathrm{Sym}^*(\mathrm{Sym}^{* > 0}(V)))$$
as long as $n \geq 2q+3$. Using $H_\bQ^{\otimes q} = \bigoplus_{\vert \lambda \vert=q} S_\lambda(H_\bQ) \otimes S^\lambda$, we obtain
$$H^q(\Aut(F_n) ; H_\bQ^{\otimes q}) \cong \mathrm{Hom}_{GL(V)}\left(\bigoplus_{\vert \lambda \vert=q}S_{\lambda'}(V) \otimes S^\lambda,\mathrm{Sym}^*(\mathrm{Sym}^{* > 0}(V))\right),$$
and using that $S^{\lambda'} \cong S^\lambda \otimes \bQ^-$ we can write the right-hand side as
$$\mathrm{Hom}_{GL(V)}\left(V^{\otimes q},\mathrm{Sym}^*(\mathrm{Sym}^{* > 0}(V))\right)\otimes \bQ^-.$$
Along with Proposition \ref{prop:IdentifySymSym} this finishes the proof of Theorem \ref{thm:Main1}.

\begin{rem}\label{rem:OddConvention}
We could have used $K(V^*, 3)$ instead of $K(V^*, 2)$. In this case the argument goes through, the analogue of Lemma \ref{lem:SymSplit} is $\mathrm{Sym}^q(H_\bQ \otimes V) \cong \bigoplus_{\vert \lambda \vert = q} S_\lambda(H_\bQ) \otimes S_{\lambda}(V)$, and the result obtained is
$$H^q(\Aut(F_n) ; S_\lambda(H_\bQ)) \cong \mathrm{Hom}_{GL(V)}(S_{\lambda}(V),S^*(\wedge^{* > 0}(V)))$$
whenever $\dim_\bQ V \geq q$ and $n \geq 2q+3$, where $S^*(-)$ denotes the free graded-commutative algebra.

A consequence of this is that the multiplicity of $S_\lambda(V)$ in $S^*(\wedge^{* > 0}(V))$ is the same as the multiplicity of $S_{\lambda'}(V)$ in $\mathrm{Sym}^*(\mathrm{Sym}^{*>0}(V))$, which does not seem obvious to the author.
\end{rem}

\subsection{Proof of Theorem \ref{thm:Main2}}\label{sec:PfB}

We will make use of the following lemma, which follows from Kawazumi \cite[Theorem 7.1]{KawazumiMagnus}. It also follows from general principles: the Becker--Gottlieb transfer with local coefficients.

\begin{lem}\label{lem:OutSplit}
For any $\bZ[\tfrac{1}{n-1}][\mathrm{Out}(F_n)]$-module $M$, the map
$$H^*(\mathrm{Out}(F_n);M) \lra H^*(\mathrm{Aut}(F_n);M)$$
is split injective, and
$$H^*(\mathrm{Aut}(F_n);M) \cong H^*(\mathrm{Out}(F_n);M) \oplus H^{*-1}(\mathrm{Out}(F_n);H^* \otimes M).$$
\end{lem}

Theorem \ref{thm:Main2} (i) and the first part of Theorem \ref{thm:Main2} (ii) follows immediately from this lemma, as it implies that $H^i(\Out(F_n);S_\lambda(H^*_\bQ))$ and $H^i(\Out(F_n);S_\lambda(H_\bQ))$ are summands of $H^i(\Aut(F_n);S_\lambda(H^*_\bQ))$ and $H^i(\Aut(F_n);S_\lambda(H_\bQ))$ respectively, so vanish under the stated assumptions. 

It remains to prove the second part of Theorem \ref{thm:Main2} (ii). To do this, we again consider a finite-dimensional $\bQ$-vector space $V$ and consider the diagram
\begin{equation*}
\xymatrix{
K(H^* \otimes_\bZ V^*,1) \ar[d]\\
\mathrm{map}(\vee^n S^1, K(V^*, 2)) \ar[d]^p \ar[r] & \mathcal{G}_n(K(V^*, 2))  \ar[r]^-{\pi} &\mathcal{G}_n(*) \simeq B\mathrm{Out}(F_n)\\
K(V^*, 2)
}
\end{equation*}
where the row is a (split) fibration and the column is a (trivial) fibration. This gives a spectral sequence
\begin{equation}\label{eq:sseqOutV}
H^*(\Out(F_n) ; \wedge^*(H_\bQ \otimes V[1])) \otimes \mathrm{Sym}^*(V[2]) \Longrightarrow H^*(\mathcal{G}_n(K(V^*, 2));\bQ).
\end{equation}

Rather than a weight argument, which is no longer conclusive, we will deduce the degeneration of this spectral sequence from the vanishing results already established, namely that $H^i(\Out(F_n); S_\lambda(H_\bQ))=0$ in degrees $2i \leq n-\vert \lambda \vert-3$ if $i \neq \vert \lambda \vert$. By Lemma \ref{lem:SymSplit} this implies that $H^p(\Out(F_n) ; \wedge^q(H_\bQ \otimes V[1]))=0$ as long as $2p \leq n-q-3$ and $p \neq q$, so in bidegrees $(p,q)$ such that $2p \leq n-q-3$ the $E^2$-page of the spectral sequence \eqref{eq:sseqOutV} is
$$\bigoplus_{\substack{a \geq 0\\ b \geq 0}} H^a(\Out(F_n) ; \wedge^a(H_\bQ \otimes V[1])) \otimes \mathrm{Sym}^b(V[2]),$$
where the $(a, b)$th summand lies in $E^2_{a, a+2b}$. In particular, all summands in this range lie in even total degree, and so there can be no differentials.

\begin{lem}\label{lem:OutResult}
There is an isomorphism of graded $GL(V)$-representations
$$\bigoplus_{\lambda} H^*(\Out(F_\infty) ; S_\lambda(H_\bQ)) \otimes S_{\lambda'}(V[1])  \cong \mathrm{Sym}^*(\mathrm{Sym}^{*>1}(V[2])),$$
and $\dim_\bQ H^{\vert \lambda \vert}(\Out(F_n) ; S_\lambda(H_\bQ))$ is independent of $n$ for $n \geq 4\vert \lambda \vert+3$.
\end{lem}
\begin{proof}
Writing $\wedge^a(H_\bQ \otimes_\bQ V[1]) \cong \bigoplus_{\vert\mu\vert =a} S_\mu(H_\bQ) \otimes S_{\mu'}(V[1])$ using Lemma \ref{lem:SymSplit}, and then identifying the abutment of the spectral sequence \eqref{eq:sseqOutV} as in Lemma \ref{lem:Abut}, in total degree $2r$ we obtain 
$$\bigoplus_{a+b=r}\bigoplus_{\vert\mu\vert=a} H^a(\Out(F_n) ; S_\mu(H_\bQ)) \otimes S_{\mu'}(V) \otimes \mathrm{Sym}^b(V) \cong [\mathrm{Sym}^*(\mathrm{Sym}^{*>0}(V[2]))]_{2r}$$
as long as $4r \leq n-3$.

The right-hand side is independent of $n$. For a partition $\lambda \vdash r$ we find that
$$\sum_{\substack{a+b=r\\ \vert\mu\vert=a}}  \dim_\bQ H^a(\Out(F_n) ; S_\mu(H_\bQ)) \cdot \dim_\bQ \mathrm{Hom}_{GL(V)}(S_{\lambda'}(V), S_{\mu'}(V) \otimes \mathrm{Sym}^b(V))$$
is independent of $n$ as long as $n \geq 4\vert \mu\vert+3$. By induction we may assume that all terms with $\vert \mu \vert < r$ are also independent of $n$, leaving just the terms with $b=0$, which gives $\dim_\bQ H^{\vert \lambda \vert}(\Out(F_n) ; S_\lambda(H_\bQ))$ by Schur's lemma. This proves the second part.
 
Passing to the limit $ n \to \infty$ and summing over all terms above gives an isomorphism of graded $GL(V)$-representations
$$\bigoplus_{\lambda} H^*(\Out(F_\infty) ; S_\lambda(H_\bQ)) \otimes S_{\lambda'}(V[1]) \otimes \mathrm{Sym}^*(V[2]) \cong \mathrm{Sym}^*(\mathrm{Sym}^{*>0}(V[2])).$$
Because the graded $GL(V)$-representation
$$\mathrm{Sym}^*(V[2]) = \bQ \oplus \mathrm{Sym}^1(V[2]) \oplus \mathrm{Sym}^2(V[2])  \oplus \cdots $$
is the trivial representation in grading zero, the isomorphism in the lemma can be established by induction on degree similarly to the argument above. More conceptually, the graded virtual representation $\mathrm{Sym}^*(V)$ is invertible under $\otimes$, as it is the trivial representation in grading zero, and cancelling it from both sides gives the required isomorphism.
\end{proof}

Proceeding as in the proof of Theorem \ref{thm:Main1}, we obtain
$$H^q(\Out(F_n);H_\bQ^{\otimes q}) \cong \mathrm{Hom}_{GL(V)}(V^{\otimes q}, \mathrm{Sym}^*(\mathrm{Sym}^{*>1}(V)))\otimes \bQ^-$$
from which Proposition \ref{prop:IdentifySymSym} implies Theorem \ref{thm:Main2}.

\subsection{Proof of Theorem \ref{thm:Main4}}

By Remark 2 after Theorem 4.2.2 of \cite{Manivel}, there is an identity
$$\bigoplus_\lambda \nu^\infty(\lambda) S_\lambda(V)  \cong \mathrm{Sym}^*(\mathrm{Sym}^{*>1}(V))$$
of $GL(V)$-modules. Compare with Lemma \ref{lem:OutResult} for the first statement; the second statement follows from this and \eqref{eq:OutAutTrade}.

\section{Integral and torsion calculations}\label{sec:integral}

The general technique we have been using is not confined to rational coefficients. To give an example of how it may be used more generally, we shall now develop a ``tame" strengthening of Theorem \ref{thm:Main1}, namely Theorem \ref{thm:Main5} (ii). Let us first explain how Theorem \ref{thm:Main5} (ii) implies Theorem \ref{thm:Main5} (i).

\begin{proof}[Proof of Theorem \ref{thm:Main5} (i)]
By Theorem \ref{thm:Main1} we have that $H^1(\Aut(F_\infty);H_\bQ) = \bQ$ and all other rational cohomology groups vanish. By Theorem \ref{thm:Main5} (ii), for each prime number $p$, $H^*(\Aut(F_\infty);H \otimes \bZ_{(p)})$ is a free $H^*(\Aut(F_\infty);\bZ_{(p)})$-module, so by the calculation above is a free module on a single generator in degree 1. Hence $H^1(\Aut(F_n);H)$ is $\bZ$, as it becomes a free $\bZ_{(p)}$-module of rank 1 when localised at any prime $p$. Choosing a generator for this group gives a map
$$H^*(\Aut(F_n);\bZ) \lra H^{*+1}(\Aut(F_n);H)$$
which must be an isomorphism, as it is so when localised at every prime.
\end{proof}

In order to prove Theorem \ref{thm:Main5} (ii) we first revisit some of the techniques we have used earlier and develop them in the $p$-local rather than rational setting.

Let $q$ be fixed and $p$ be a prime number such that $p > q$. Recall that to a partition $\lambda \vdash q$ and a tableau $T$ of shape $\lambda$ there is an associated \emph{Young symmetriser} $c_T \in \bZ[\Sigma_q]$, which satisfies $c_T^2 = \frac{q!}{\dim_\bQ S^\lambda} c_T$. As $q!$ is a $p$-local unit we may form the element $e_T := \frac{\dim_\bQ S^\lambda}{q!} c_T \in \bZ_{(p)}[\Sigma_q]$ and this is an idempotent. If $g \in \Sigma_q$ then $e_{gT} = g e_T g^{-1}$, so there is a conjugacy class of idempotent associated to each $\lambda$. We write $e_\lambda$ for the idempotent associated to the canonical tableau for $\lambda$. The Specht module is defined by $S^\lambda := \bQ[\Sigma_q]\cdot e_\lambda$, and by analogy we define the $p$-local Specht module by $S^\lambda_{(p)} := \bZ_{(p)}[\Sigma_q]\cdot e_\lambda$. It is a free $\bZ_{(p)}$-module of rank $\dim_\bQ S^\lambda$, and is indecomposable as a $\bZ_{(p)}[\Sigma_q]$-module (as $S^\lambda_{(p)} \otimes \bQ = S^\lambda$). These idempotents satisfy $e_T \cdot e_S = 0$ if $T$ and $S$ are tableaux of different shapes. Furthermore we have a decomposition into primitive idempotents
$$1 = \sum_{\lambda \vdash q} \sum_{\substack{\text{standard tableaux}\\ \text{$T$ of shape $\lambda$}}} e_T \in \bZ_{(p)}[\Sigma_q].$$

If $M$ is a $\bZ_{(p)}$-module then $M^{\otimes q}$ is a $\bZ_{(p)}[\Sigma_q]$-module, and we define the $p$-local Schur functor by $S_\lambda(M) := e_\lambda(M^{\otimes q})$. There is then a natural map
$$\phi : \bigoplus_{\lambda \vdash q}  S^\lambda_{(p)} \otimes_{\bZ_{(p)}} S_\lambda(M) \lra M^{\otimes q}$$
of $\bZ_{(p)}[\Sigma_q]$-modules given by $(x \cdot e_\lambda) \otimes e_\lambda(m_1 \otimes \cdots \otimes m_q)  \mapsto x \cdot e_\lambda(m_1 \otimes \cdots \otimes m_q)$. For an inverse, define
$$\psi(m_1 \otimes \cdots \otimes m_q) = \sum_{\lambda \vdash q} \sum_{\substack{\text{standard tableaux}\\ \text{$T$ of shape $\lambda$}}} (g_T \cdot e_\lambda) \otimes e_\lambda(g_T^{-1}(m_1 \otimes \cdots \otimes m_q))$$
where for a tableau $T$ of shape $\lambda$ we write $e_T = g_T e_\lambda g_T^{-1}$. This establishes the Schur--Weyl decomposition in this setting.

We now require a partial analogue of the notion of \emph{weights} which we used in the proof of Lemma \ref{lem:collapse}. For a $\bZ_{(p)}$-module $M$ and an integer $t$, write $M(t)$ for the $\bZ_{(p)}[\bZ_{(p)}^\times]$-module which is the same as a $\bZ_{(p)}$-module and on which $u \in \bZ_{(p)}^\times$ acts as scalar multiplication by $u^t$; say that it is a module which is \emph{pure of weight $t$}. 

\begin{lem}\label{lem:LocalWeight}
If $M$ and $N$ are  $\bZ_{(p)}$-modules then
$$\mathrm{Ext}_{\bZ_{(p)}[\bZ_{(p)}^\times]}^*(M(t), N(t'))=0$$
if $0 < \vert t - t'\vert < p-1$.
\end{lem}
\begin{proof}
By the natural isomorphisms $X(t) \otimes_{\bZ_{(p)}[\bZ_{(p)}^\times]} \bZ_{(p)}(s) \cong X(t+s)$ it is enough to establish the lemma for $t'=0$. In this case consider the functor
$$F_{M,t}(N) := \mathrm{Hom}_{\bZ_{(p)}[\bZ_{(p)}^\times]}(M(t), N(0)) : \bZ_{(p)}\text{-Mod} \lra \mathbf{Ab}.$$
When $M=\bZ_{(p)}$ this satisfies
$$F_{\bZ_{(p)}, t}(N) = \{x \in N \, \vert \, (u^t-1)x=0 \text{ for all } u \in \bZ_{(p)}^{\times}\}.$$
By \cite[Lemma 2.12]{JXII} the $\gcd$ of the numbers $k^t-1$ over all integers $k$ coprime to $p$ is itself coprime to $p$ as long as $0 < \vert t \vert < p-1$, and hence $F_{\bZ_{(p)}, t}(N)=0$ under this assumption on $t$. The functor $N \mapsto N(0)$ is exact, so taking derived functors of $F_{\bZ_{(p)}, t}$ shows that the claim in the lemma holds for $M=\bZ_{(p)}$, and hence for $M$ any free $\bZ_{(p)}$-module. The claim in general follows by resolving $M$ by free modules.
\end{proof}

We now give the proof of Theorem \ref{thm:Main5} (ii), which states that under our assumptions on $p$ and $q$, $H^*(\Aut(F_\infty);S_\lambda(H_{(p)}))$ is a free $H^*(\Aut(F_\infty); \bZ_{(p)})$-module.

\begin{proof}[Proof of Theorem \ref{thm:Main5} (ii)]
Let $k$ be an odd integer, and let $S^k_{(p)}$ be a choice of model for the $p$-local $k$-sphere. For each unit $u \in \bZ_{(p)}^\times$ we may find a map $f_u : S^k_{(p)} \to S^k_{(p)}$ inducing multiplication by $u$ on $H_k(S^k_{(p)};\bZ)$, and these satisfy $f_u \circ f_v \simeq f_{u\cdot v}$. Set $Y := (S^k_{(p)})^q$, and consider the space $\mathcal{G}_n^1(Y)$, the Serre fibration
\begin{equation}\label{eq:LocalFibn}
\map_*(\vee^n S^1, Y) \lra \mathcal{G}_n^1(Y) \lra \mathcal{G}_n^1(*) \simeq B\Aut(F_n),
\end{equation}
and its associated Serre spectral sequence.

Recall that the $\bZ_{(p)}$-cohomology ring of $\Omega S^k_{(p)}$ is the divided power algebra $\Gamma_{\bZ_{(p)}}^*(x_{k-1})$ on the class obtained by looping a generator of $H^k(S^k;\bZ_{(p)})$. This identifies the $\bZ_{(p)}$-cohomology ring of $\map_*(\vee^n S^1, Y)$ with $\Gamma_{\bZ_{(p)}}^*(H_{(p)}[k-1])^{\otimes q}$, so, taking the limit $n \to \infty$, there is a spectral sequence
$$E_2^{*,*}=H^*(\Aut(F_\infty) ; \Gamma_{\bZ_{(p)}}^*(H_{(p)}[k-1])^{\otimes q}) \Longrightarrow H^*(Q_0(S^0) \times QY;\bZ_{(p)}).$$

It follows from the results of \cite[p.\ 40]{CLM} that the map 
$$\mathrm{Sym}^*_{\bZ_{(p)}}(\widetilde{H}_*(Y;\bZ_{(p)})) \lra H_*(QY;\bZ_{(p)})$$
from the free commutative algebra, induced by the map $Y \to Q(Y)$, is an isomorphism on homology in degrees $* \leq pk$. By the K{\"u}nneth theorem we have
$$H_*(Y;\bZ_{(p)}) = (\bZ_{(p)}[0] \oplus \bZ_{(p)}[k])^{\otimes q},$$
which is free as a $\bZ_{(p)}$-module. This shows that $H^*(QY;\bZ_{(p)})$ is a free $\bZ_{(p)}$-module in degrees $* \leq pk$, and shows that it is supported in degrees divisible by $k$. Using Galatius' theorem and the K{\"u}nneth theorem again we may therefore identify the target of the spectral sequence with $H^*(\Aut(F_\infty);\bZ_{(p)}) \otimes_{\bZ_{(p)}} H^*(QY;\bZ_{(p)})$ in a range of degrees. In total degrees $* \leq pk$ the spectral sequence takes the form
$$E_2^{*,*}=H^*(\Aut(F_\infty) ; \Gamma_{\bZ_{(p)}}^*(H_{(p)}[k-1])^{\otimes q}) \Rightarrow H^*(\Aut(F_\infty);\bZ_{(p)}) \otimes H^*(QY;\bZ_{(p)})$$
and is a spectral sequence of $H^*(\Aut(F_\infty);\bZ_{(p)})$-modules. 

For $u \in \bZ_{(p)}^\times$ the homotopy equivalence $f_u^q : Y \to Y$ induces a map of spectral sequences by functoriality, making it into a spectral sequence of $\bZ_{(p)}[\bZ_{(p)}^\times]$-modules. The induced map on $\Gamma_{\bZ_{(p)}}^i(H_{(p)}[k-1])^{\otimes q}$ is given by scalar multiplication by $u^{i q}$, so the $j(k-1)$st row of the spectral sequence is pure of weight $j$. Furthermore, the fibration \eqref{eq:LocalFibn} has a section so there are no differentials entering the bottom row, and this row splits off the filtration. The rows $(k-1),2(k-1), \ldots, (p-1)(k-1)$ have different weights which differ by at most $(p-2)$, so by Lemma \ref{lem:LocalWeight} there are no differentials in total degree $* < p(k-1)$ and the associated filtration of $H^*(\Aut(F_\infty);\bZ_{(p)}) \otimes H^*(QY;\bZ_{(p)})$ splits as $\bZ_{(p)}[\bZ_{(p)}^\times]$-modules.

The map induced by $f_u^q$ on $H^{kj}(QY;\bZ_{(p)})$ is multiplication by $u^j$, so this is pure of weight $j$. It follows that the $q(k-1)$st row of the spectral sequence may be identified in a range of degrees with $H^*(\Aut(F_\infty);\bZ_{(p)}) \otimes H^{kq}(QY;\bZ_{(p)})$ and so is a free $H^*(\Aut(F_\infty);\bZ_{(p)})$-module. The $\Aut(F_n)$-module $\Gamma_{\bZ_{(p)}}^*(H_{(p)}[k-1])^{\otimes q}$ contains $(H_{(p)}[k-1])^{\otimes q}$ as a summand in degree $q(k-1)$, so contains $S_\lambda(H_{(p)})$ as a summand in this degree too. Thus the $q(k-1)$st row of the spectral sequence contains $H^*(\Aut(F_\infty) ; S_\lambda (H_{(p)}))$ as a summand, so this is a projective $H^*(\Aut(F_\infty);\bZ_{(p)})$-module in degrees $* < p(k-1)-q(k-1) = (p-q)(k-1)$, so in all degrees as $p>q$ and $k$ was arbitrary. Finally, as $H^*(\Aut(F_\infty);\bZ_{(p)})$ is a connected graded algebra over the local ring $\bZ_{(p)}$, projective graded modules which are finitely-generated in each degree are free.
\end{proof}

\appendix

\section{An integral calculation for $\mathrm{Out}(F_n)$}\label{sec:conj}

It seems reasonable to suppose that Theorem \ref{thm:Xstab} (ii) holds with integral and not just $\bZ[\tfrac{1}{n-1}]$-module coefficients, that is, that

\begin{conj}\label{conj:LastPtStab}
The map $\mathcal{G}^1_n(X) \to \mathcal{G}_{n}(X)$ induces an isomorphism on homology in degrees $2* \leq n - 3$.
\end{conj}

Putting this together with Theorem \ref{thm:Xlim}, it follows that $\tau^0_n : \mathcal{G}_{n}(X) \to Q_{1-n}(X_+)$ is an isomorphism on integral homology in degrees $2* \leq n - 3$. Assuming this conjecture, we may make the following calculation. 

\begin{prop}\label{prop:IntCalcOut}
We have
\begin{align*}
H^1(\mathrm{Out}(F_n);H) &= 0 \quad\quad\quad\quad \text{\,\,\,\,for $n \geq 7$}\\
H^2(\mathrm{Out}(F_n);H) &= \bZ/(n-1) \quad \text{for $n \geq 9$.}
\end{align*}
\end{prop}

We emphasise that this proposition holds \emph{without} assuming Conjecture \ref{conj:LastPtStab}: it follows from Theorem \ref{thm:Main2} that these groups are torsion, Satoh \cite{Satoh} has computed that $H_1(\Out(F_n);H^*) \cong \bZ/(n-1)$ for $n \geq 4$, and one easily computes that $H_0(\Out(F_n);H^*)=0$. Our purpose here is to give another proof of this proposition using Conjecture \ref{conj:LastPtStab}.

This proposition should be contrasted with with a theorem of Bridson and Vogtmann \cite[Theorem B]{BridVogt}, who show that the extension
$$H = F_n / F_n' \lra \mathrm{Aut}(F_n)/F_n' \lra \mathrm{Out}(F_n)$$
is non-trivial for all $n \geq 2$, and hence gives a non-trivial class $\zeta \in H^2(\mathrm{Out}(F_n);H)$. Our calculation $H^2(\mathrm{Out}(F_n);H) = \bZ/(n-1)$ along with the result of Bridson and Vogtmann \cite{BridVogt} that their class $\zeta$ remains non-trivial in the group $H^2(\mathrm{Out}(F_n);H/rH)$ for any $r$ not coprime to $(n-1)$ implies that the class $\zeta$ generates $H^2(\mathrm{Out}(F_n);H)$ as long as $n \geq 9$.

\begin{proof}[Proof of Proposition \ref{prop:IntCalcOut} assuming Conjecture \ref{conj:LastPtStab}]
Consider the diagram
\begin{equation*}
\xymatrix{
BH^* \ar[d]\\
\mathrm{map}(\vee^n S^1, \bC\bP^\infty) \ar[d]^p \ar[r] & \mathcal{G}_n(\bC\bP^\infty)  \ar[r]^-{\pi} &\mathcal{G}_n(*) \simeq B\mathrm{Out}(F_n)\\
\bC\bP^\infty
}
\end{equation*}
where the row and column are fibrations. Note that $p$ is a trivial fibration and is split via the inclusion $s : \bC\bP^\infty \to \mathrm{map}(\vee^n S^1, \bC\bP^\infty)$ of the constant maps, and there is an inclusion $\iota : \mathcal{G}_n(*) \times \bC\bP^\infty \to \mathcal{G}_n(\bC\bP^\infty)$ of the graphs with constant maps to $\bC\bP^\infty$. The Leray--Serre spectral sequence for the horizontal fibration is
\begin{equation}\label{eq:sseqOut}
\bar{E}_2^{p,*} := H^p(\mathrm{Out}(F_n); \wedge^* H) \otimes \bZ[a] \Longrightarrow H^*(\mathcal{G}_n(\bC\bP^\infty);\bZ)
\end{equation}
where $a$ is the canonical class in $H^2(\bC\bP^\infty;\bZ)$, so has bidegree $(p,q)=(0,2)$.

We first claim that the map
$$\bC\bP^\infty \overset{s}\lra \mathrm{map}(\vee^n S^1, \bC\bP^\infty) \lra \mathcal{G}_n(\bC\bP^\infty)$$
has image $(n-1)\bZ \subset \bZ = H^2(\bC\bP^\infty;\bZ)$ on second cohomology. By our conjecture, it is enough to prove this after composing with the map $\tau_n^0 : \mathcal{G}_n(\bC\bP^\infty) \to Q_0(\bC\bP^\infty_+)$ as long as $n \geq 7$. Up to translation of components, this map is given by the Becker--Gottlieb transfer $\bC\bP^\infty  \to Q_{1-n}(\bC\bP^\infty \times (\vee^n S^1)_+)$ for the trivial graph bundle over $\bC\bP^\infty$ composed with projection to $Q_{1-n}(\bC\bP^\infty_+)$. By standard properties of the transfer, this is $(1-n)$ times the standard inclusion, which on second cohomology induces multiplication by $(1-n)$, as required.

This describes the edge homomorphism of the spectral sequence (\ref{eq:sseqOut}). We now use the homotopy equivalence
$$Q_0(\bC\bP^\infty_+) \simeq Q_0(S^0) \times Q(\bC\bP^\infty)$$
and that $S^2 \to Q(\bC\bP^\infty)$ is 3-connected (by the Freudenthal suspension theorem) to describe the cohomology of $Q_0(\bC\bP^\infty_+)$ in low degrees. Using that $H^i(Q_0(S^0);\bZ)$ is $\bZ, 0, \bZ/2, \bZ/2$ for $i=0,1,2,3$, it follows that $H^i(Q_0(\bC\bP^\infty_+);\bZ)$ is $\bZ, 0, \bZ \oplus \bZ/2, \bZ/2$ for $i=0,1,2,3$. As the spectral sequence (\ref{eq:sseqOut}) converges to zero for positive Leray filtration in total degree $3$, the differential $d_2 : \bZ = \bar{E}_2^{0, 2} \to \bar{E}_2^{2,1}$ must be onto (so $\bar{E}_2^{2,1}$ is cyclic) and the kernel is $(n-1)\bZ$, so $\bar{E}_2^{2,1} = H^2(\mathrm{Out}(F_n);H) \cong \bZ/(n-1)$. On the other hand, in total degree $2$ we see $(n-1)\bZ = \bar{E}_\infty^{0,2}$ and $\bZ/2 = \bar{E}_\infty^{2,0}$, and it converges to $\bZ/2 \oplus \bZ$, so observing the direction of the Leray filtration we see that $H^1(\mathrm{Out}(F_n);H) = \bar{E}_2^{1,1} =0$.
\end{proof}

\section{Mapping class groups of surfaces}\label{sec:MCG}

Let $\Gamma_{g}$ denote the mapping class group of a surface $\Sigma_{g}$ of genus $g$, and consider $H_\bQ := H_1(\Sigma_{g};\bQ)$ as a $\Gamma_{g}$-module. In this case Poincar{\'e} duality gives $H^*_\bQ \cong H_\bQ$. Looijenga has already computed $H^*(\Gamma_{g};S_\lambda(H_\bQ))$ in the stable range \cite{Looijenga2}, but we wish to explain here how the computation may also be performed using the techniques of this paper.

Cohen and Madsen \cite{CM} have introduced spaces $\mathcal{S}_{g}(X)$ of surfaces diffeomorphic to $\Sigma_{g}$ equipped with a map to $X$, and have identified the stable (co)homology of these spaces with that of the infinite loop space $\Omega^\infty(\mathrm{MTSO}(2) \wedge X_+)$ as long as $X$ is simply-connected. Furthermore, there is a fibration sequence
$$\map(\Sigma_g, X) \lra \mathcal{S}_g(X) \lra B\Gamma_g$$
so an associated Serre spectral sequence.

Considering $X=K(V^*, 4)$, the argument of Section \ref{sec:PfB} goes through without change and identifies
$$\bigoplus_{\lambda} H^*(\Gamma_{g} ; S_\lambda(H_\bQ)) \otimes S_{\lambda'}(V)[3\vert\lambda\vert] \otimes \mathrm{Sym}^*(V[2]) \otimes \mathrm{Sym}^*(V[4])$$
with
$$H^*(\Omega^\infty_0(\mathrm{MTSO}(2) \wedge K(V^*, 4)_+);\bQ) = \mathrm{Sym}^*([W_* \otimes \mathrm{Sym}^*(V[4])]_{>0})$$
in a range of degrees increasing with $g$, where $W_* = H^*(\mathrm{MTSO}(2);\bQ)$ is the graded vector space which is $\bQ$ in degrees $-2,0,2,4,6,\ldots$ and zero otherwise. This may be written as
$$\mathrm{Sym}^*(W_{*>0}) \otimes \mathrm{Sym}^*(V[2]) \otimes \mathrm{Sym}^*(V[4]) \otimes \mathrm{Sym}^*(W_{*>0} \otimes V[4] \oplus W_* \otimes \mathrm{Sym}^{*>1}(V[4]))$$
and the first term $\mathrm{Sym}^*(W_{*>0})$ is $H^*(\Gamma_{g};\bQ)$ in the stable range, so we obtain
$$\bigoplus_{\lambda} H^*(\Gamma_{g} ; S_\lambda(H_\bQ)) \otimes S_{\lambda'}(V)[3\vert\lambda\vert] \cong H^*(\Gamma_g;\bQ) \otimes \mathrm{Sym}^*(W_{*>0} \otimes V[4] \oplus W_* \otimes \mathrm{Sym}^{*>1}(V[4]))$$
in a range of degrees increasing with $g$.

It follows that each $H^*(\Gamma_{g} ; S_\lambda(H_\bQ))$ is a free $H^*(\Gamma_{g};\bQ)$-module in the stable range. To describe the space of module generators we can proceed as in the proof of Theorem \ref{thm:Main1}, and hence identify $H^*(\Gamma_g; H_\bQ^{\otimes q})$ as a graded $\Sigma_q$-module with
$$H^*(\Gamma_g;\bQ) \otimes \mathrm{Hom}_{GL(V)}(V^{\otimes q}, \mathrm{Sym}^*(W_{*>0} \otimes V[4] \oplus W_* \otimes \mathrm{Sym}^{*>1}(V[4])))[-3q] \otimes \bQ^-.$$
Proposition \ref{prop:IdentifySymSym} identifies the space of $GL(V)$-module homomorphisms  with the permutation module on the set of the following data: a partition $P$ of $\{1,2,\ldots,q\}$, a labeling of each part of size 1 in the set $\{x_2, x_4, x_6, \ldots\}$, and a labeling of each part of size $>1$ in the set $\{x_{-2}, x_0, x_2, x_4, x_6, \ldots\}$. Such a datum is given grading $q$ plus the sum of the degrees of the labels (which are given by their subscripts).

\begin{example}
When $q=1$ we find that $H^*(\Gamma_g;H_\bQ)$ is a free $H^*(\Gamma_g;\bQ)$-module with generators in degrees $\{3,5,7, \ldots\}$.

When $q=2$ the partition $\{1,2\}$ may be labeled by $\{x_{-2}, x_0, x_2, x_4, x_6, \ldots\}$, and the partition $\{\{1\},\{2\}\}$ may have each part labeled by $\{x_2, x_4, x_6, \ldots\}$. As a $\Sigma_2$-set the action is trivial on the first type of elements, and the second type form trivial orbits
$$\{\{1;x_{2i}\},\{2;x_{2i}\}\}$$
and free orbits
$$\{\{1;x_{2i}\},\{2;x_{2j}\}\}, \{\{1;x_{2j}\},\{2;x_{2i}\}\}$$
with $i \neq j$. Thus $H^*(\Gamma_g;\wedge^2(H_\bQ))$ has $H^*(\Gamma_g;\bQ)$-module generators given by the multiplicities of the trivial representation in the indicated permutation module, so in degrees $\{0,2,4,6,\ldots;6,8,10,\ldots; 10,12,14,\ldots; 14,16,18,\ldots;\ldots\}$. Similarly $H^*(\Gamma_g;\mathrm{Sym}^2(H_\bQ))$ has $H^*(\Gamma_g;\bQ)$-module generators given by the multiplicities of the sign representation in the indicated permutation module, so in degrees $\{8,10,12,\ldots; 12,14,16,\ldots;16,18,20,\ldots;\ldots\}$.

See \cite[Theorem G]{ERW10} for a complete calculation in the case of exterior powers, and \cite[Proposition 5.2]{RWImm} for a complete calculation in the case of symmetric powers.
\end{example}

By considering the analogous space $\mathcal{S}_{g,1}(X)$ of surfaces of genus $g$ with one boundary equipped with maps to $X=K(V^*, 4)$ which map the boundary to the basepoint, one identifies $\bigoplus_{\lambda} H^*(\Gamma_{g,1} ; S_\lambda(H_\bQ)) \otimes S_{\lambda'}(V)[3\vert\lambda\vert]$ with
$$ H^*(\Gamma_{g,1};\bQ) \otimes \mathrm{Sym}^*(W_{*>-2} \otimes V[4] \oplus W_* \otimes \mathrm{Sym}^{*>0}(V[4]))$$
in a range of degrees increasing with $g$, from which $H^*(\Gamma_{g,1} ; H_\bQ^{\otimes q})$ may be computed as above. There is a similar description in terms of partitions, the only difference being that now pieces of a partition of size 1 may also be labeled by $x_0$. For example, one finds that $H^*(\Gamma_{g,1} ; H_\bQ)$ is a free $H^*(\Gamma_{g,1};\bQ)$-module with generators in degrees $\{1,3,5,7,\ldots\}$. This description of $H^*(\Gamma_{g,1} ; H_\bQ^{\otimes q})$ (in terms of labelled partitions) has previously been obtained by Kawazumi \cite{Kawazumi}, cf.\ Theorem 1.B of that paper, even with integral as opposed to just rational coefficients. However Kawazumi does not establish this description \emph{as $\Sigma_q$-modules}, because his method of proof breaks the $\Sigma_q$-symmetry.

\section{Tables}\label{sec:tables}

The result of Theorems \ref{thm:Main1} and \ref{thm:Main2} for $\vert \lambda \vert \leq 6$ is compiled in the tables below, which was computed with the {\tt SchurRings} package for {\tt Macaulay2} \cite{M2}, and checked via Theorem \ref{thm:Main4} and the table at the end of Section 4.3 of \cite{Manivel}.

\begin{table}[h!]
  \centering
  %\label{tab:table1}
  \caption*{Dimensions of $H^{\vert \lambda \vert}(\Aut(F_n) ; S_\lambda(H_\bQ))$ for $\vert \lambda \vert \leq 6$ and $n \geq 2\vert \lambda \vert+3$.}
  \begin{tabular}{c|c|cc|ccc|ccccc}
    \toprule
    $()$& $(1)$ & $(1^2)$ & $(2)$ & $(1^3)$ & $(21)$ & $(3)$ & $(1^4)$ & $(2 1^2)$ & $(2^2)$ & $(31)$ & $(4)$\\
    \midrule
    1 & 1 & 2 & 0 & 3 & 1 & 0 & 5 & 2 & 2 & 0 & 0\\
    \bottomrule
  \end{tabular}
  \begin{tabular}{ccccccc}
    \toprule
    $(1^5)$ & $(2 1^3)$ & $(2^2 1)$ & $(3 1^2)$ & $(32)$ & $(41)$ & $(5)$\\
    \midrule
    7 & 5 & 4 & 0 & 1 & 0 & 0 \\
    \bottomrule
  \end{tabular}
  \begin{tabular}{ccccccccccc}
    \toprule
    $(1^6)$ & $(2 1^4)$ & $(2^2 1^2)$ & $(2^3)$ & $(3 1^3)$ & $(321)$ & $(3^2)$ & $(4 1^2)$ & $(42)$ & $(51)$ & $(6)$\\
    \midrule
    11 & 8 & 10 & 2 & 1 & 2 & 2 & 0 & 0 & 0 & 0 \\
    \bottomrule
  \end{tabular}
\end{table}

\begin{table}[h!]
  \centering
  %\label{tab:table1}
  \caption*{Dimensions of $H^{\vert \lambda \vert}(\Out(F_n) ; S_\lambda(H_\bQ))$ for $\vert \lambda \vert \leq 6$ and $n \geq 4\vert \lambda \vert+3$.}
  \begin{tabular}{c|c|cc|ccc|ccccc}
    \toprule
    $()$ & $(1)$ & $(11)$ & $(2)$ & $(111)$ & $(21)$ & $(3)$ & $(1111)$ & $(211)$ & $(22)$ & $(31)$ & $(4)$\\
    \midrule
    1 & 0 & 1 & 0 & 1 & 0 & 0 & 2 & 0 & 1 & 0 & 0\\
    \bottomrule
  \end{tabular}
  \begin{tabular}{ccccccc}
    \toprule
    $(11111)$ & $(2111)$ & $(221)$ & $(311)$ & $(32)$ & $(41)$ & $(5)$\\
    \midrule
    2 & 1 & 1 & 0 & 0 & 0 & 0 \\
    \bottomrule
  \end{tabular}
  \begin{tabular}{ccccccccccc}
    \toprule
    $(1^6)$ & $(2 1^4)$ & $(2^2 1^2)$ & $(2^3)$ & $(3 1^3)$ & $(321)$ & $(3^2)$ & $(4 1^2)$ & $(42)$ & $(51)$ & $(6)$\\
    \midrule
    4 & 1 & 3 & 0 & 0 & 0 & 1 & 0 & 0 & 0 & 0\\
    \bottomrule
  \end{tabular}
\end{table}

\bibliographystyle{plain}
\bibliography{MainBib}

\def\cprime{$'$}
\begin{thebibliography}{10}

\bibitem{JXII}
J.~F. Adams.
\newblock On the groups {$J(X)$}. {II}.
\newblock {\em Topology}, 3:137--171, 1965.

\bibitem{BridVogt}
Martin~R. Bridson and Karen Vogtmann.
\newblock Abelian covers of graphs and maps between outer automorphism groups
  of free groups.
\newblock {\em Math. Ann.}, 353(4):1069--1102, 2012.

\bibitem{Brown}
Kenneth~S. Brown.
\newblock {\em Cohomology of groups}, volume~87 of {\em Graduate Texts in
  Mathematics}.
\newblock Springer-Verlag, New York, 1994.
\newblock Corrected reprint of the 1982 original.

\bibitem{CE}
Henri Cartan and Samuel Eilenberg.
\newblock {\em Homological algebra}.
\newblock Princeton University Press, Princeton, N. J., 1956.

\bibitem{CLM}
Frederick~R. Cohen, Thomas~J. Lada, and J.~Peter May.
\newblock {\em The homology of iterated loop spaces}.
\newblock Springer-Verlag, Berlin, 1976.
\newblock Lecture Notes in Mathematics, Vol. 533.

\bibitem{CM}
Ralph Cohen and Ib~Madsen.
\newblock Surfaces in a background space and the homology of mapping class
  group.
\newblock {\em Proc. Symp. Pure Math.}, 80(1):43--76, 2009.

\bibitem{CM2}
Ralph Cohen and Ib~Madsen.
\newblock Stability for closed surfaces in a background space.
\newblock {\em Homology Homotopy Appl.}, 13(2):301--313, 2011.

\bibitem{DjamentHodge}
Aur{\'e}lien Djament.
\newblock D{\'e}composition de {H}odge pour l'homologie stable des groupes
  d'automorphismes des groupes libres.
\newblock arXiv:1510.03546, 2015.

\bibitem{DV}
Aur{\'e}lien Djament and Christine Vespa.
\newblock Sur l'homologie des groupes d'automorphismes des groupes libres \`a
  coefficients polynomiaux.
\newblock {\em Comment. Math. Helv.}, 90(1):33--58, 2015.

\bibitem{DwyerHomStab}
W.~G. Dwyer.
\newblock Twisted homological stability for general linear groups.
\newblock {\em Ann. of Math. (2)}, 111(2):239--251, 1980.

\bibitem{ERW10}
Johannes Ebert and Oscar Randal-Williams.
\newblock Stable cohomology of the universal {P}icard varieties and the
  extended mapping class group.
\newblock {\em Doc. Math.}, 17:417--450, 2012.

\bibitem{galatius-2006}
S{\o}ren Galatius.
\newblock Stable homology of automorphism groups of free groups.
\newblock {\em Ann. of Math. (2)}, 173(2):705--768, 2011.

\bibitem{GR-W}
S{\o}ren Galatius and Oscar Randal-Williams.
\newblock Monoids of moduli spaces of manifolds.
\newblock {\em Geom. Topol.}, 14(3):1243--1302, 2010.

\bibitem{M2}
Daniel~R. Grayson and Michael~E. Stillman.
\newblock Macaulay2, a software system for research in algebraic geometry.
\newblock Available at {\tt http://www.math.uiuc.edu/Macaulay2}.

\bibitem{HV2}
Allen Hatcher and Karen Vogtmann.
\newblock Cerf theory for graphs.
\newblock {\em J. London Math. Soc. (2)}, 58(3):633--655, 1998.

\bibitem{HV}
Allen Hatcher and Karen Vogtmann.
\newblock Homology stability for outer automorphism groups of free groups.
\newblock {\em Algebr. Geom. Topol.}, 4:1253--1272 (electronic), 2004.

\bibitem{Ivanov}
Nikolai~V. Ivanov.
\newblock On the homology stability for {T}eichm\"uller modular groups: closed
  surfaces and twisted coefficients.
\newblock In {\em Mapping class groups and moduli spaces of {R}iemann surfaces
  ({G}\"ottingen, 1991/{S}eattle, {WA}, 1991)}, volume 150 of {\em Contemp.
  Math.}, pages 149--194. Amer. Math. Soc., Providence, RI, 1993.

\bibitem{KawazumiMagnus}
Nariya Kawazumi.
\newblock Cohomological aspects of {M}agnus expansions.
\newblock arXiv:math/0505497, 2005.

\bibitem{Kawazumi}
Nariya Kawazumi.
\newblock On the stable cohomology algebra of extended mapping class groups for
  surfaces.
\newblock In {\em Groups of diffeomorphisms}, volume~52 of {\em Adv. Stud. Pure
  Math.}, pages 383--400. Math. Soc. Japan, Tokyo, 2008.

\bibitem{Looijenga2}
Eduard Looijenga.
\newblock Stable cohomology of the mapping class group with symplectic
  coefficients and of the universal {A}bel-{J}acobi map.
\newblock {\em J. Algebraic Geom.}, 5(1):135--150, 1996.

\bibitem{Manivel}
L.~Manivel.
\newblock Gaussian maps and plethysm.
\newblock In {\em Algebraic geometry ({C}atania, 1993/{B}arcelona, 1994)},
  volume 200 of {\em Lecture Notes in Pure and Appl. Math.}, pages 91--117.
  Dekker, New York, 1998.

\bibitem{Procesi}
Claudio Procesi.
\newblock {\em Lie groups}.
\newblock Universitext. Springer, New York, 2007.

\bibitem{R-WAutOld}
Oscar Randal-Williams.
\newblock The stable cohomology of automorphisms of free groups with
  coefficients in the homology representation.
\newblock arXiv:1012.1433, 2010.

\bibitem{RWImm}
Oscar Randal-Williams.
\newblock The space of immersed surfaces in a manifold.
\newblock {\em Math. Proc. Cambridge Philos. Soc.}, 154(3):419--438, 2013.

\bibitem{R-WW}
Oscar Randal-Williams and Nathalie Wahl.
\newblock Homological stability for automorphism groups.
\newblock arXiv:1409.3541, 2014.

\bibitem{Satoh}
Takao Satoh.
\newblock Twisted first homology groups of the automorphism group of a free
  group.
\newblock {\em J. Pure Appl. Algebra}, 204(2):334--348, 2006.

\bibitem{Satoh2}
Takao Satoh.
\newblock Twisted second homology groups of the automorphism group of a free
  group.
\newblock {\em J. Pure Appl. Algebra}, 211(2):547--565, 2007.

\bibitem{VespaExt}
Christine Vespa.
\newblock Extensions between functors from groups.
\newblock arXiv:1511.03098, 2015.

\end{thebibliography}

\end{document}